\documentclass[]{amsart}
\usepackage{epsfig}
\usepackage{amsmath, amssymb}
\numberwithin{equation}{section}

\newtheorem{theorem}{Theorem}[section]
\newtheorem{lemma}[theorem]{Lemma}
\newtheorem{corollary}[theorem]{Corollary}
\newtheorem{proposition}[theorem]{Proposition}

\newtheorem{definition}{Definition}[section]

\allowdisplaybreaks \numberwithin{equation}{section}

\newcommand{\weglassen}[1]{}

\renewcommand{\imath}{\mathrm{i}}

\begin{document}
\title{$p$-adic generalizations of Hyper-elliptic $\lambda$ functions}
\begin{abstract}We express the branch points cross ratio of Hyper-elliptic Mumford curves as quotients of $p$ adic theta functions 
evaluated at the $p$ adic period matrix
\end{abstract}
\author{ Yaacov Kopeliovich}
\maketitle
\section{introduction}
\footnote{This note woouldn't have been written had it not been for Jeremy Teitelbaum. I thank him for posing the question and his encouragement to pursue it.} 
In this note we generalize $\lambda$ functions of Hyper-elliptic curves that are well known in the complex case,\cite{FK980} to the $p$ adic case. That is given a $p$-adic Mumford hyper-elliptic curve of genus $g$ with the equation: 
\begin{equation}
y^2=x(x-1)\prod_{i=1}^{2g-1}(x-\lambda_i) 
\end{equation} 
we derive formulas expressing branch points of hyper-elliptic curves $\lambda_i$ ( more precisely their cross ratios) for the $p$-adic $\lambda$ as quotients of $p$-adic theta functions.  
The question of obtaining such formulas is a natural question following the paper of \cite{T}. In this paper the author calculated $p$-adic periods for modular curves $X_0(p)$ for certain primes $p$ such that the genus of $X_0(p)$ was 2. The calculation of the periods involved the following 3 steps: 
\begin{enumerate} 
\item The author used $p$ adic theta functions to obtain explicit formulas for $\lambda_1,\lambda_2,\lambda_3$
\item He used these formulas to obtain $\lambda_i$ as infinite series of the elements of the $p-$ adic period matrix 
\item He inverted these series to obtain an expression of the periods as a function of $\lambda_1,\lambda_2,\lambda_3.$
\item He applied those to verify the exceptional zero conjecture for these $X_0(p).$
\end{enumerate} 
We focus on the first part and  obtain formulas for $\lambda_1,...\lambda_{2g-1}$ for arbitrary $g,$ generalizing the formulas obtained in \cite{T} to any genus.
While the derivation of this part in \cite{T} followed a detailed analysis of the fundamental $p$-adic group action on the associated tree and its fundamental domain our analysis will follow the classical approach that was exemplified in \cite{FK980} and \cite{M2}. This approach enables us to write the quotients of participating theta functions instantly provided we know that certain divisors are non-special. In the hyper-elliptic case ( and more generally in the cyclic cover case) we can characterize these non-special divisors that are supported on the branch points of the covering completely. Thus the only task that is left is to compute the $p$-adic characteristics of the images of divisors. We accomplish following the work of \cite{V1},\cite{V2}. 

This note is divided into three sections. In the first section we collect the relevant facts about Mumford curves and $p$ adic theta functions. In the second note we calculate the images of the relevant characteristics of the Jacobian and in the last section we prove our main theorem.

\section{Recap on Mumford curves} 
In this section we gather the relevant facts about Hyper-elliptic Mumford curves above non-Archimedean field. The reader should consult \cite{GV} and \cite{V1} and the references there for proofs of the assertions in this section.Let $k$ a complete non-Archimedean field which is algebraically closed. Let $s_0,...s_g$ in $PGL(2,k)$ be elements of order $2,$ such that the group $\Gamma_0$ generated by them satisfies: 
\begin{enumerate} 
\item $\Gamma_0$ is discontinuous 
\item $\Gamma_0$ is the free product of the groups generated by the elements $s_i.$
\end{enumerate} 
 The kernel of the homomorphism induced $\phi:\Gamma_0\mapsto \{\pm 1\}$ given by $\phi(s_i)=-1$ for all $i$ is called the Whitakker group. $\Gamma$ is a free group on the generators $s_1s_0,...s_gs_0.$ We have the following: 
\begin{proposition}
The groups $\Gamma$,$\Gamma_0$ have the same set of ordinary points denoted by $\Omega.$ The curve $X=\Omega/\Gamma$ is isomorphic analytically to a hyper-elliptic curve $X$  and the curve $\Omega/\Gamma_0$ is isomorphic to $\mathbb{P}^{1}(k).$ Furthermore the mapping $\Pi^{an}:X\mapsto \mathbb{P}^{1}(k)$ is: 
 $x\mod \Gamma=x\mod \Gamma_0.$  
\end{proposition} 
\begin{proof} See \cite{GV} and \cite{V1}.\end{proof}
We denote the automorphism of the hyper-elliptic curve by $\overline{s_0}$. 

Let us briefly recall how the Jacobian of the curve is defined. If $H$ is schottky group of rank $g+1$ and $X_H=\Omega/H$ is the corresponding Mumford curve. Recall that, $\Omega\subset\mathbb{P}^{1}(k)$ is the set of the ordinary points of $H.$ For each $a,b\in \Omega$ define: 
\begin{equation} 
u_{a,b}(z)=\prod_{h\in H}\frac{z-ha}{z-hb}
\end{equation} 
This product defines a meromorphic function on $\Omega$ which satisfies the equation: $c_{a,b}(h)u_{a,b}(z)=u_{a,b}(hz).$ If $b$ doesn't belong to the orbit of $a,$ $h(z)$ has zeros precisely on the orbit of $a$ and poles in the orbit of $b$. If $b=ha,h\in H$  $u_{a,b}$ doesn't depend on $a$ and we will denote these mappings by $u_h$ and $c_h$ respectively. $u_h$ has no zeros or poles. 
For a Schottky group let $G_H=Hom(H,k^{\star}).$ Because $H$ is free we can identify the latter group with $\left({k^{\star}}\right)^{g+1}$ The group $\Lambda_{H}=\left\{c_h|h\in H\right\}$ is a free Abelian group of rank $g+1.$ which is discrete in $G_H.$ Assume $D=\sum_{i=1}^n (a_i-b_i)$ is a divisor on $X_H.$ The mapping:
$$\sum_{i=1}^n(a_i-b_i)\mapsto \prod_{i=1}^n c_{a_i,b_i}$$ induces a mapping from $J(H)$ and the quotient $G_H/\Lambda_H.$ To define the corresponding Jacobian mapping let $p\in \Omega$ define $t_H:X_H\mapsto G_H$ by $t(x)=c_{x,p}.$ This is the canonical embedding that extends to the mapping on divisors. The dual variety $\widehat{J_H}$ can be represented as : $\widehat{G_H}/\widehat{\Lambda_H}$. $\widehat{G_H}=Hom(G_H,k^{\star})$ and $\widehat{\Lambda_H}=\left\{d\in \widehat{G_H}|\exists \alpha\in H, d(c_\gamma)=c_{\alpha}(\gamma)\right\}$
Define the action of  $\Lambda_H$  on $\mathbb{O}^{\star}(G_H\}=\left\{f|f, \text{ holomorphic and nowhere vanishing on, } G_H\right\}$ by $f^{c_\gamma}(c)=f(c_\gamma c).$ Let $\xi\in Z^{1}(\Lambda_H,O^{\star}(G_H))$ be a one cocycle and let $$\L_\xi=\left\{g | g \text{ holomorphic function on } G_H, g(c)=\xi_{c_{\gamma}}(c)g(c_\gamma c),\forall c_\gamma\in \Lambda_H\right\}.$$ 
\begin{definition} 
Elements of $L(\xi)$ are called theta functions of type $\xi.$
\end{definition} 
Now we construct the basic theta function. Let $p_H:\Lambda_H\times \Lambda_H$ be a symmetric bilinear form such that $p_H^2(c_\gamma,c_\delta)=c_\gamma(\delta),\forall \gamma,\delta \in H.$ Define the one cocycle by 
$$\xi_{H,c_\gamma}(c)=p_H(c_\gamma,c_\gamma)c(\gamma),c_\gamma\in \Lambda_H,c\in G_H.$$ In this case, $dim(L_\xi)=1.$ and generated by the Riemann theta function: 
\begin{equation} 
\theta_H(c)=\sum_{c_\gamma\in \Lambda_H}\xi_{H,c_\gamma}(c).
\end{equation}
The divisor of this function is $\Lambda_H$ invariant and hence induces a divisor on $J_H.$ This divisor defines a polarization $\Theta_H$ on $J_H.$ Riemann theorem is valid in the $p$-adic case. More precisely we have that: 
\begin{theorem}\textbf{Riemann Vanishing Theorem} 

\begin{enumerate} 
\item The Holomorphic function $\theta_H\circ t_H$ has an $H$-invariant divisor which regarded as a divisor on $X_H,$ has degree of $g+1$ 
\item If the map $\overline{t_H}:X_H\mapsto J_H$ is based on the point $p\in \Omega$ and if 
$K_H=(div(\theta\circ t_H \mod H\in Div(X_H),$ then $2K_H$ is a canonical divisor. Furthermore the class of $K_H$ under linear equivalence of divisors doesn't depend on choice of $p.$
\item if $c\in G_H$ then $\theta_{H}(c)=0$ if and only if $\overline{c}=\overline{t_H}(D-K_H)$ for some positive divisor $D$ of degree $g.$ The order of vanishing of $t_H$ at $c$ equals to $i(D)$ the index of specialty of $D.$ 
\end{enumerate} 
\end{theorem} 
\begin{proof}
See \cite{V2} and \cite{GV}
\end{proof}
\section{The images of the branch points in the Jacobian} 
We calculate the explicit images of the branch points.
Assume that  $s_0,...s_g$ are generators of $\Gamma_0$ of order $2.$ Then the fixed points of $s_i$ are ordinary points $a_i,b_i.$ Let $\gamma_i=s_is_0,i=1...g$ Our goal is to show the following theorem: 
\begin{theorem} We have the following equatlities:

\begin{enumerate} 
\item $c_{a_0,b_0}(\gamma_i)=-1;i=0,...g$ \label{first}
\item $c_{a_i,a_0}^2=c_{b_i,a_0}^2=c_{\gamma_i}$\label{second}
\item $c_{b_ia_0}=c_{b_ia_i}c_{a_ia_0}$\label{third}
\item $c_{b_ia_i}(\gamma_j)=(-1)^{\delta_{ij}}, i,j\geq 1$ and $\delta_{ij}$ is the Kronecker delta. \label{fourth}
\end{enumerate} 
\end{theorem} 
 \begin{proof} 
 We reproduce the proof of the theorem following \cite{V2} in several stages first we show the following lemma: 
\begin{lemma}
Assume $\alpha\in N(\Gamma)$ then $u_{a,b}(\alpha z)=cu_{\alpha^{-1}(a),\alpha^{-1}(b)}(z)$ $c$ depends on $a,b,\alpha.$ 
\end{lemma}
\begin{proof} \textbf{Proof of lemma:} For $\alpha$ as above we let 
\[\alpha=
   \begin{pmatrix} 
      p & q \\ 
      r & s 
   \end{pmatrix}
\]
First assume the $\det\alpha=1$ then we have that \[\alpha^{-1}=
   \begin{pmatrix} 
      s & -q \\ 
      -r & p 
   \end{pmatrix}
\]
\begin{equation} 
u_{a,b}(\alpha(z))=\prod_{\gamma\in \Gamma}\frac{\alpha(z)-\gamma a}{\alpha(z)-\gamma b}=\prod_{\gamma\in\Gamma}
=\prod_{\gamma\in \Gamma}\frac{\frac{pz+q}{rz+s}-\gamma a}{\frac{pz+q}{rz+s}-\gamma b}
=\prod_{\gamma\in \Gamma}\frac{pz+q-	rz\gamma a-s\gamma a}{pz+q-rz\gamma b- s\gamma b }
\end{equation}   
but $z-\alpha^{-1}\gamma a = z-\frac{s\gamma a-q}{-r\gamma a+p}=\frac{-r\gamma az+pz-s\gamma a+q}{-r\gamma a +p}$
Using the last equality we can write the product as: 
$$\prod_{\gamma\in \Gamma}\frac{z-\alpha^{-1}\gamma a }{a-\alpha^{-1} \gamma b }\prod_{\gamma\in \Gamma}\frac{-\gamma a+p}{-r\gamma b + p }$$
Call $c=\prod_{\gamma\in\Gamma}\frac{-r\gamma a+p}{-r\gamma b + p }$ as it doesn't depend on $z.$ Hence we have : 
$\prod_{\gamma\in \Gamma}\frac{\alpha(z)-\gamma a}{\alpha(z)-\gamma b}=c\frac{z-\alpha^{-1}\gamma a }{a-\alpha^{-1} \gamma b }$ Now use the fact that $\alpha^{-1}\gamma\alpha\in \Gamma$ to conclude the lemma. 
 \end{proof}
 We have the following corollary from the lemma: 
\begin{corollary}
$u_{\gamma}(\alpha z)=cu_{\alpha^{-1}\gamma\alpha}(z).$
\end{corollary}  
 \begin{proof}\textbf{of corollary} 
By definition we can choose any $a$ and then we have that: 
\begin{equation}
u_{\alpha a,\gamma \alpha a}(\alpha z) = cu_{ a,\alpha^{-1}\gamma\alpha a}(z)
\end{equation} 
Conclude the corollary from the last equality and the definition of $u_\gamma.$ 
\end{proof}
To continue the proof of the main theorem note that for any two ramification points of the cover $p_,p_j$ we have that $c_{p_i,p_j}$ must be of order $2$ in the Jacobian since we have functions whose divisor is $2P_i-2P_j.$ Indeed if $\lambda_i,\lambda_j$ are the images of these points take the function $\frac{x-\lambda_i}{x-\lambda_j}.$ For $\infty$ we can take the function $x-\lambda_i.$
Let us show the parts of the theorem: 
\begin{proof}\textbf{Proof of part (\ref{first})}
First observe that: $c_{a_0,b_0}(\gamma_i)=\frac{u_{a_0,b_0}(s_0s_ia_i)}{u_{a_0,b_0}(s_0a_i}=c_{a_0,b_0}(s_0)$
But because $s_0^2=1$ we must have that $c_{a_0,b_0}^2=1.$ If $c_{a_0,b_0}=1$ we have that $u_{a_0,b_0}$is $\Gamma$ invariant and induces a meromorphic function on $X_{\Gamma}.$ This function has exactly one zero and one pole and implies that $g(X_\Gamma=0$ a contradiction. 

\textbf{Proof of part(\ref{fourth})}
The previous proof gives us a stronger result that is, $c_{b_ia_i}(s_is_j)=-1,i\neq j$ We conclude immediately that  $c_{b_i,a_i}(s_0s_i)=-1$ if $i\neq j$ we use the fact that $s_i^2=1$ and thus $s_0s_j=s_0s_1s_1s_j$ hence:  $c_{b_i,a_i}(s_0s_j)=c_{b_i,a_i}(s_0s_is_is_j)=c_{b_i,a_i}(s_0s_i)c_{b_i,a_i}(s_0s_j)=(-1)\times(-1)=1$
\end{proof}
\textbf{Proof of part (\ref{second})}
By definition we have that: 
\begin{equation}
c_{a_i,a_0}=\frac{u_\gamma (a_i)^2}{c_{\gamma} (a_0)^2}=\frac{u_\gamma(a_i)u_{\gamma}(s_i^{-1}a_i)}{u_\gamma (s_0^{-1}a_0)}=\frac{u_\gamma(s_0^{-1}\gamma_i)}{u_{\gamma} (s_0^{-1}a_0)}\frac{u_{\gamma}(a_i)}{u_{\gamma}(a_0)}
\end{equation}
By definition of $c_{\gamma}$ we rewrite the last expression as : 
\begin{equation} 
c_{\gamma}(s_0^{-1}\gamma_i^{-1}s_0^{-1})\frac{u_\gamma(a_i)u_{s_0^{-1}\gamma s_0}(a_i)}{u_{\gamma}(a_0)u_{s_0^{-1}\gamma s_0}(a_0)}=c_{\gamma}(\gamma_i)c_{a_i,a_0}(s_0^{-1}\gamma s_0^{-1}\gamma_0^{-1})
\end{equation}
But $s_0^{-1}\gamma s_0^{-1}\gamma_0^{-1}=1$ and because $c_\gamma(\delta)=c_{\delta}(\gamma)$ we complete the proof of part \ref{second}
\end{proof}
We end this section by calculating explicitly Riemann's constant in the $p$ adic case follwoing \cite{V2}. Choose a polarization such that $p_\Gamma(c_{\gamma_i,c_{\gamma_i})}=c_{a_i,a}.$ For this polarization we form the theta function as above: 
\begin{equation} 
\theta_\Gamma(c)=\sum_{c_\gamma\in \Lambda_{\Gamma}}\xi_{\Gamma,c_\gamma}(c).
\end{equation}
and $\xi_{\Gamma,c_\gamma}$ is the cocycle associated with the polarization we defined above. 
We have the following lemma that enables us to determine the zeros of the $\theta$ function we just defined. 
\begin{lemma}\label{coollemma} 
Let $c\in G_{\Gamma}$ such that  $c^2=c_\gamma\in \Lambda_\Gamma$ with $\gamma\not\in[\Gamma,\Gamma]$and such that $c(\gamma)=-p_{\Gamma}(c_\gamma,c_\gamma).$ Then $\theta_{\Gamma}(c)=0.$
\end{lemma}
\begin{proof}\textbf{Proof of \ref{coollemma}}
We have that: 
\begin{equation}
\theta_{\Gamma}(c)=\theta_{\Gamma}(c^{-1}c_\gamma)=\xi^{-1}_{\Gamma,c_\gamma}(c^{-1})\theta_{\Gamma}(c^{-1})
\end{equation}  
But $\xi^{-1}_{\Gamma,c_\gamma}(c^{-1})=p_\Gamma(c_\gamma,c_\gamma)\times c(\gamma)^{-1}=-1$ because $\theta_\Gamma(c)$ is an even function the assertion follows. 
\end{proof} 
We apply the last lemma to $c_{a_0,a_i}$ and $c_{a_0,b_i}$ to obtain the following corollary: 
\begin{corollary}
Under the choice of the polarization the zeros of $\theta_{\Gamma}$ are the points $b_i.$
\end{corollary}
We note that more generally for any polarization we have $c_{b_i,a}=-c_{a_i,a}=\pm p_\Gamma\left(c_{\gamma_i},c_{\gamma_i}\right)$ thus the zeros of $\theta_{\Gamma,p}$ correspond to the points $a_i,b_i.$
We will show in the next section that any divisor of the form $a_{i_1}...a_{i_k},b_{i_{k+1}}...b_{i_g}$ and $a_{i_k},b_{i_l}$ are distnict is a non-special divisor. Combining this with Riemann's theorem we get the following: 
\begin{theorem} 
Under the choice of polarization above $K_{\Gamma}=\sum_{i=1}^kb_g.$ in the $J(X).$ 
\end{theorem} 
\section{Uniformizations of Hyper-elliptic curves} 
We calculate the degree of vanishing For any given a divisor $\zeta$ of degree $g$ supported on the branch points $\left\{a_1,...a_g,b_1...b_{g-1}\right\}$. We use that to obtain uniformizations of the cross ratio generalizing the expressions in the usual case of complex analysis. As our theorem won't distinguish $a_1,...a_g,b_1,...b_g$ we denote them by $P_1,...P_{2g}$ and assume that $P_1=x^{-1}(\infty)$ and $P_2=x^{-1}(0).$
\begin{theorem} \label{mmm}
Let $r,s$ be a non-negative integers such that $2s+r=g-1.$ Then for every choice of branch points $P_{i_1}...P_{i_r},P_{j_1},...P_{j_s}$ we have that $i(2\sum_{k=1}^rP_{i_k}\sum_{l=1}^sP_{j_l})=s.$
\end{theorem} 
\begin{proof}\textbf{Proof of Theorem 4.1}.
Let $U=2\sum_{k=1}^rP_{i_k}\sum_{l=1}^sP_{j_l}.$ Apply the R.R theorem:  
Using the Riemann Roch theorem we have : $i(U)=g-1-degU+r(U^{-1}).$ In our case we have : 
$i(U)=r(U^{-1})-1$ and hence it is enough to show that $r(U^{-1})=s+1.$ The divisor $U^{-1}$ is a divisor that is invariant under the action of the automorphism hence we can decompose this $\mathcal{O}(U^{-1})$ into a direct sum: $\mathcal{O}(-U)=V_0(-U)\bigoplus V_1(-U).$ Where $V_0,V_1$ are the eigenspaces of eigenvalues $\pm{1}.$  Now if $f\in V_1,$ $\frac{f}{y}$ is also invariant under the action of the automorphism and thus is isomorphic to $V_0(-U-div y)$ and $V_0(\mathcal{O}((-U-div y)$ is the subspace of $\mathcal{O}(-(U-div y))$ fixed by the automorphism of the hyper-elliptic curve. Now if $D=\sum_{a_i}P_i,a_i=0,1$ $\pi$ induces an isomorphism between $V_0(\mathcal(O)(-D)$ and $\mathcal(O)(-D_0)$ and $D_0=\sum_{i}a_i\pi(P_i)$ and $P_i$ are the branch points. Now apply Riemann Roch to $\mathbb{P}^{1}(k).$ to conclude the result. 
\end{proof}
We have the following corollary 
\begin{corollary} 
Let $D=\sum_{j=1}^k\overline{a_{i_j}}\sum_{m=1}^l \overline{b_{i_m}}$, $k+l=g(X),i_j,i_m\neq 0 $ then $D$ is a non-special divisor. 
\end{corollary}  
We are going to use the last corollary in the last section to obtain uniformizations of Hyper-elliptic curves generalizing the work at \cite{T}. Let us consider two seqeunces of $\mathbf{P_1},\mathbf{P_2}$ of the set $1,...2g$ such that they will differ exactly in one element. For example if we have a genus $2$ curve we can choose: 
$\mathbf{P_1}=\left\{1,2\right\},\mathbf{P_2}=\left\{1,3\right\}.$ To ease notation let us switch from $a_i,b_i$ notation to $P_i$ notation. In this case $a_0,a_1,...a_g$ will be denoted by the even indices i.e. they will correspond to the points $P_0,P_2...P_{2g+2}$ and $b_1,...b_g$ correspond to the points $P_1,P_3,...P_{2g+1}.$ So in this case $P_0$ is the base of our mapping and $P_3...P_{2g+1}$ corresponds to $K_{P_0}.$ We will also require that the sequences corresponding to $\mathbf{P_1},\mathbf{P_2}$ will not be equal to the sequence $P_3P_5...P_{2g+1}.$ Now let us define: $\theta_{\Gamma,\mathbf{P_j}}(c)=\theta_\Gamma(c-\sum_{P_i\in \mathbf{P_j}}P_i+K_{P_0}).$
$\forall P\in X$ Consider the mapping: 
\begin{equation} 
P\mapsto \frac{\theta_{\Gamma,\mathbf{P_1}}^2(t_\Gamma(P))}{\theta_{\Gamma,\mathbf{P_2}}^2(t_\Gamma(P))}
\end{equation} 
According to Riemann's theorem the zero and poles of these functions are: 
$$
\frac{2\sum_{P_i\in \mathbf{P_1}}P_i}{2\sum_{P_k\in \mathbf{P_2}}P_k}=\frac{2P_l}{2P_m}
$$
and $P_l,P_m$ are the unique points which are different in $\mathbf{P_1}$ and $\mathbf{P_2}.$ Hence we have that: 
$\frac{\theta_{\Gamma,\mathbf{P_1}}^2(t_\Gamma(P))}{\theta_{\Gamma,\mathbf{P_2}}^2(t_\Gamma(P))}=C\frac{z(P)-\lambda_l}{z(P)-\lambda_m}$
To find the constant choose any $P_h$ such that $P_h\not\in \mathbf{P_1}\cup \mathbf{P_2}.$ we have that:   
$\mathbf{P_1}$ and $\mathbf{P_2}.$ Hence we have that: 
$$\frac{\theta_{\Gamma,\mathbf{P_1}}^2(t_\Gamma(P_h))}{\theta_{\Gamma,\mathbf{P_2}}^2(t_\Gamma(P_h))}=C\frac{\lambda_h-\lambda_l}{\lambda_h-\lambda_m}$$ Hence 
$$
\frac{\lambda_h-\lambda_m}{\lambda_h-\lambda_l}\frac{\theta_{\Gamma,\mathbf{P_1}}^2(t_\Gamma(P_h))}{\theta_{\Gamma,\mathbf{P_2}}^2(t_\Gamma(P_h))}=C
$$
Hence we obtained the following theorem: 
\begin{theorem} 
For any partitions $\mathbf{P_1},\mathbf{P_2}$ as above let the pre-images of $\lambda_l,\lambda_m$ be the unique points such that $\lambda_l\not\in \mathbf{P_2}$ and $\lambda_m\not\in \mathbf{P_1}.$ assume $\lambda_h,\lambda_k$ any points such that $\lambda_h,\lambda_l\not\in \mathbf{P_1}\cup\mathbf{P_2}.$ We have that: 
\begin{equation}
\frac{\theta_{\Gamma,\mathbf{P_1}}^2(t_\Gamma(P_m))^2}{\theta_{\Gamma,\mathbf{P_2}}^2(t_\Gamma(P_m))^2}\frac{\theta_{\Gamma,\mathbf{P_2}}^2(t_\Gamma(P_h))}{\theta_{\Gamma,\mathbf{P_1}}^2(t_\Gamma(P_h))^2}
=\frac{\lambda_h-\lambda_m}{\lambda_h-\lambda_l}\frac{\lambda_k-\lambda_m}{\lambda_k-\lambda_l}
\end{equation}
\end{theorem} 
Note that in this case we can calculate precisely the vector that corresponds to these partitions. Indeed for each subset $\mathbf{P}\subset \left\{1,...2g+2\right\}$ define $\delta_P=\sum_{i\in \mathbf{P}}\delta_{ij}.$ We can define an element $c_P$ in $Hom\left(\Gamma,K^{star}\right)$ by $c_P=\sum_{i\in P}\gamma_i.$ Using the definition of $c_P$ we rewrite the theorem: 
Hence we obtained the following theorem: 
\begin{theorem} 
For any partitions $\mathbf{P_1},\mathbf{P_2}$ as above let the pre-images of $\lambda_l,\lambda_m$ be the unique points such that $\lambda_l\not\in \mathbf{P_2}$ and $\lambda_m\not\in \mathbf{P_1}.$ assume $\lambda_h,\lambda_k$ any points such that $\lambda_h,\lambda_l\not\in \mathbf{P_1}\cup\mathbf{P_2}.$ Define $\mathbf{P_{is}}=\mathbf{P_i}\cup \lambda_s.$ Denote by $\mathbf{O}$ the set of odd numbers $2j+1,1\leq j\leq g$ and let $\mathbf{O}_{is}=\mathbf{P_{is}}\triangle\mathbf{O}.$ ($\triangle$ is the symmetric difference of two sets) 
\begin{equation}
\frac{\theta_\Gamma (c_{\mathbf{O}_{1m}})^2\theta_\Gamma(c_{\mathbf{O}_{2h}})^2}{\theta_\Gamma (c_{\mathbf{O}_{2m}})^2\theta_\Gamma(c_{\mathbf{O}_{1h}})^2}
=\frac{\lambda_h-\lambda_m}{\lambda_h-\lambda_l}\frac{\lambda_k-\lambda_m}{\lambda_k-\lambda_l}
\end{equation}
\end{theorem} 
\subsection{conclusion}
In this note we generalized the formulas that are well known in the classical case for the expression of the cross ratio of the branch points of Hyper-elliptic curves through theta functions. These formulas were previously known for $g=2$ \cite{T}. We replace the fundamental domain considerations of the action of $PGL_2(K)$ where $K$ is a non-Archimedean field with direct computation as in \cite{V1}. In principle we believe that these formulas can be applied in the spirit of \cite{T} to calculate the periods of all the Hyper-elliptic modular curves for higher genuses. Remarkably \cite{V1} calculated the images of the branch points for Mumford curves that are cyclic covers. Thus similar formulas should be available in that case too and further generalizaitons of  more general Thomae type formulas. We intend to pursue it in future works.

\bibliographystyle{plain}

\end{document}